%% file: main.tex
\documentclass[12pt,reqno]{amsart}
\usepackage{amsmath, amsthm, amssymb, stmaryrd}
\usepackage{cleveref}
\usepackage{bbm}

\crefformat{section}{\S#2#1#3} 
\crefformat{subsection}{\S#2#1#3}
\crefformat{subsubsection}{\S#2#1#3}

\advance \topmargin by -\headheight
\advance \topmargin by -\headsep
     
\evensidemargin \oddsidemargin
\input{abbrv.tex}

\begin{document}
\title[Weighted Quartic Forms]{A Quantitative Hasse Principle for Weighted Quartic Forms}
\author[Daniel Flores]{Daniel Flores}
\address{Department of Mathematics, Purdue University, 150 N. University Street, West 
Lafayette, IN 47907-2067, USA}
\email{flore205@purdue.edu}
\date{}
\dedicatory{}
\begin{abstract}
We derive, via the Hardy-Littlewood method, an asymptotic formula for the number of integral zeros of a particular class of weighted quartic forms under the assumption of non-singular local solubility. Our polynomials $F(\bfx,\bfy) \in \Z[x_1,\ldots,x_{s_1},y_1,\ldots,y_{s_2}]$ satisfy the condition that $F(\lambda^2 \bfx, \lambda \bfy) = \lambda^4 F(\bfx,\bfy)$. Our conclusions improve on those that would follow from a direct application of the methods of Birch. For example, we show that in many circumstances the expected asymptotic formula holds when $s_1 \ge 2$ and $2s_1 + s_2 > 8$.
\end{abstract}
\maketitle

\section{Introduction}\label{sec:intro}
Historically, the Hardy-Littlewood method has been largely successful when employed to show the existence of simultaneous zeros to families of diagonal polynomials. The most famous applications being those devoted to Waring's problem and the ternary Goldbach problem. This is mainly due to the well-developed nature of the theory of exponential sums in one variable. More generally, there has been progress toward applying the circle method to show the existence of simultaneous zeros to families of forms. Most of these results stem from the work of Birch \cite{Birch1962}, in which he uses the geometry of numbers to obtain an analogue of Weyl's inequality in the homogeneous setting with the additional restriction that the dimension of the singular locus is not too large. In these results the number of variables required grows exponentially with the degree of these forms. In this paper we consider a slightly more general case in which the polynomials of interest are integral weighted forms.
\begin{definition}\label{def1.1}
    We define a polynomial $F \in \Z[x_1,\ldots,x_s]$ to be an \emph{integral weighted form} of degree $d$ if there exists a vector $\bfw \in \N^{s}$, satisfying $(w_1,\ldots,w_s) = 1$, for which the equation 
    \[F(\lambda^{w_1} x_1,\ldots, \lambda^{w_s} x_s) = \lambda^d F(x_1,\ldots,x_s),\] 
    holds for all complex $\lambda$. With this notation we say that the variable $x_i$ has weight $w_i$.
\end{definition}

We aim to leverage variables with large weight to achieve additional cancellation over the minor arcs. This allows us to reduce the number of variables needed to employ the circle method. In general, this problem is difficult without an analogue Weyl's inequality, such as that of Lemma 2.1 from \cite{Birch1962}, for weighted forms which takes into consideration the variables with large weight. A natural starting point would be to investigate weighted forms of low degree which are nearly diagonalizable, such that this problem may be tractable using known methods. In this paper, our focus lies on a specific family of weighted quartic forms. After establishing our main theorem, we provide a concise overview of other challenges potentially receptive to these concepts in section \ref{sec:furprob}. \par

First, we develop some notation. Let $s$ be a fixed natural number and suppose for each $1 \le k \le s$ we are given integers $0 \le i_k,j_k \le 2$ satisfying the inequality $i_k+j_k > 0$ for each $1 \le k \le s$. Due to the nature of the weighted quartic form we consider here it is necessary to develop an indexing system for our variables which groups them in a useful manner. We accomplish this by denoting
\[\bfx_k = (x_{k,1},\ldots,x_{k,i_k}) \text{ and } \bfy_k = (y_{k,1},\ldots,y_{k,j_k}).\]
 We then define the integral weighted quartic forms 
\begin{equation}\label{1.1}
    F_k(\bfx_k;\bfy_k) = 
H^{(2)}_{k}(\bfx_k) + \sum_{1 \le l \le i_k}H^{(2,l)}_{k}(\bfy_k)x_{k,l} + H^{(4)}_{k}(\bfy_k) ,
\end{equation}
where $H_k^{(2)} \in \Z[x_1,\ldots,x_{i_k}]$ and $H_{k}^{(2,l)} \in \Z[y_1,\ldots,y_{j_k}]$ are quadratic forms, and $H_k^{(4)} \in \Z[y_1,\ldots,y_{j_k}]$ are quartic forms. It will be important later to reference the coefficients of the polynomial $H_k^{(2)}(\bfx_k)$, we do this via the labeling
\begin{equation}\label{1.2}
    H_k^{(2)}(x_1,x_2) = \sum_{0 \le i \le 2}a_{k,i}x_1^i x_2^{2-i},
\end{equation}
where if $i_k = 1$ then we simply take $a_{k,1} = a_{k,2} = 0$ and similarly if $i_k = 0$ then $a_{k,0} = a_{k,1} = a_{k,2} = 0$. We finish this setup by implementing one final notational convention, we denote
\[\grX = (\bfx_1,\ldots,\bfx_s) \text{ and } \grY = (\bfy_1 , \ldots, \bfy_s).\]
Upon setting
\[s_1 = \sum_{1 \le k \le s}i_k\quad \text{ and} \quad s_2 = \sum_{1 \le k \le s}j_k,\]
one sees that $\grX$ is $s_1$-dimensional and $\grY$ is $s_2$-dimensional and we may sometimes write these in the form
\[\grX = (x_1,\ldots,x_{s_1}) \quad \text{and}\quad \grY = (y_1,\ldots,y_{s_2}).\]
We now define the integral weighted quartic form of interest to be
\begin{equation}\label{1.3}
    \Phi(\grX;\grY) =F_{1}(\bfx_1;\bfy_1) + \cdots + F_{s}(\bfx_s;\bfy_s).
\end{equation}
One sees that $\Phi$ is an integral weighted form where $s_1$ counts the number of variables with weight two and similarly $s_2$ counts the number of variables with weight one. If one were to look at the related form 
\[\Tilde{\Phi}(\grX;\grY) = \Phi(x_1^2,\ldots,x_{s_1}^2;y_1,\ldots,y_{s_2}),\]
then Birch's methods show us that supposing $\Tilde{\Phi}$ is non-singular it satisfies the \emph{smooth} Hasse principle (as in \cite{Browning2017}) as soon as $s_1+s_2$ is at least 49, this can even be improved to 30 variables by a result of Marmon and Vishe \cite{Marmon2019}. This then allows us to say that our weighted form $\Phi$ also satisfies the smooth Hasse principle as soon as $s_1+s_2 \geq 30$. We improve on this as by Theorem \ref{thrm1.2} our weighted form $\Phi$ may satisfy the smooth Hasse principle with as little as six variables.\par

Take $\grB^{(1)}$ and $\grB^{(2)}$ be cubes of equal side lengths which are at most one, in $\R^{s_1}$ and $\R^{s_2}$. We define $R(X;\Phi;\grB^{(1)},\grB^{(2)})$ to be the number of integer zeros of $\Phi$ where 
\[ \grX \in X^{1/2}\grB^{(1)} \quad \text{and} \quad \grY \in X^{1/4}\grB^{(2)}.\] 
We call a binary form $F(x,y)$ of degree $d>1$ degenerate if there exist complex numbers $\a,\b$ such that $F(x,y) = (\a x+\b y)^d$. We may then, naturally, call the polynomials $F_k$ from (\ref{1.1}) non-degenerate if every binary form used to define it is non-degenerate. With this notation in place we are now ready to state our main result.\par

\begin{theorem}\label{thrm1.2}
    Let $\Phi$ be an integral weighted quartic form of the type defined in \eqref{1.3} where for each $1 \le k \le s$ the polynomial $F_k$ is non-degenerate and the coefficient $a_{k,0}$ of $H_k^{(2)}$ is nonzero provided that $i_k >0$. Then as long as $s_1$ and $s_2$ satisfy
    \[s_2 > 16 \quad\text{when}\quad s_1 = 0,\]
    \[s_2 > 10 \quad\text{when}\quad s_1 = 1,\]
    or
    \[s_1/2+s_2/4 > 2 \quad\text{when}\quad s_1 \ge 2,\]
    then we have the asymptotic formula
    \[R(X;\Phi;\grB^{(1)},\grB^{(2)}) = X^{s_1/2 + s_2/4 - 1}\left( \sigma_{\infty}(\grB^{(1)},\grB^{(2)}) \prod_{p}\sigma_p + o(1) \right),\]
    where
    \[\sigma_p = \lim_{h \to \infty} p^{h(1-s_1-s_2)} \# \left\{(\grX,\grY) \in (\Z/p^h\Z)^{s_1+s_2}: \Phi(\grX;\grY) \equiv 0 \nmod{p^h} \right\},\]
    and
    \[\sigma_\infty(\grB^{(1)},\grB^{(1)}) = \lim_{Q \to \infty} \int_{|\b| \le Q} \int_{\grB^{(2)}} \int_{\grB^{(1)}}e(\b \Phi(\bfgam;\bfrho)) \d \bfgam \d \bfrho \d \b.\]
    Additionally, if we assume $\Phi$ has non-singular solutions locally everywhere then 
    \[\prod_{p}\sigma_p \asymp 1,\]
    and there exists a choice of cubes $\grB^{(1)}$ and $\grB^{(2)}$ for which one may show that $\sigma_\infty(\grB^{(1)},\grB^{(2)}) \asymp 1$.
\end{theorem}
One sees that when $s_1 \geq 2$ the conclusion of Theorem \ref{thrm1.2} is near-optimal with respect to the number of variables required by the Hardy-Littlewood circle method. This may be seen by looking at the case where there are no mixed terms and every coefficient is 1, that is to say
\[\Phi(\grX;\grY) = \sum_{1 \le i \le s_1}x_{i}^2+\sum_{1 \le j \le s_2}y_j^4.\]
This diagonal case is a matter of classical theory. We refer the interested reader to section 15 of \cite{Vaughan2002} for further reading. From this classical case we have by the convexity bound that the expected asymptotic formula from Theorem \ref{thrm1.2} may only be obtained via a conventional circle method approach when $s_1/2 + s_2/4 > 2.$ \par

As is the case with many applications of the circle method the main difficulty we encounter are the minor arcs, here we give a rough outline of our methods for the minor arcs. Although we assume that $s_2$ is even for the purpose of outlining our strategy, our proof also covers the case when $s_2$ is odd. We begin by using Proposition \ref{prop3.1} and Lemma \ref{lemma3.2} to essentially remove all the monomials which contain variables of mixed weights and additionally diagonalize the variables with weight two because it cost us no extra effort to do so. We may do this at the cost of losing a factor of $(\log X)^{s_1}$ which is ultimately harmless. This process, in essence, allows us to bound the minor arc contribution relevant to our problem by bounding the minor arc contribution of an associated counting function. The latter function counts the number of integral zeros of the weighted quartic form
\[\Tilde{\Phi}(\grX;\grY) = \sum_{1 \le i \le s_1}a_ix_i^2 + \sum_{1 \le j \le s_2/2}H_{j}^{(4)}(y_{2j-1},y_{2j}),\]
where for each $j$, the binary quartic form $H_j^{(4)}$ is non-degenerate and our solutions satisfy
\[\grX \in X^{1/2}\Tilde{\grB}^{(1)} \quad\text{and}\quad \grY \in X^{1/4}\Tilde{\grB}^{(2)},\]
where $\Tilde{\grB}^{(1)}$ and $\Tilde{\grB}^{(2)}$ are slight modifications of the initial boxes $\grB^{(1)}$ and $\grB^{(2)}$. At this point, we can easily get enough savings in the minor arcs if $s_1 \geq 4$. When $s_1 = 0$ a generalized version of this problem has already been looked at by Wooley \cite{Wooley1999} where his lower bound of $s_2 > 16$ agrees with ours. The interesting cases are when $1 \le s_1 \le 3$. We take care of these cases via the use of the optimal mean value estimate, Lemma \ref{lemma3.3}. \\
\textbf{Acknowledgements}: During the course of the work on this paper, the author was supported by the NSF grant DMS-2001549 under the supervision of Trevor Wooley. The author is grateful for support from Purdue University. Additionally, the author would like to thank Trevor Wooley for suggesting this area of research and for his mentorship throughout the process. The author also extends sincere appreciation to the reviewer for their detailed and constructive feedback, which significantly improved the manuscript.

\section{Preliminary Setup}\label{sec:prelim} 
Our basic parameter is $X$, an (eventually) large positive number. Whenever $\varepsilon$ appears in a statement, either implicitly or explicitly, we assert that the statement holds for every $\varepsilon>0$. In this paper, implicit constants in Vinogradov's notation $\ll$ and $\gg$ may depend on $\varepsilon$ and on the coefficients of the polynomial $\Phi$. We also make use of the vector notation $\bfx = (x_1,\ldots,x_r)$ where $r$ is dependent on the context of the argument. Whenever the notation $|\bfx|$ is used for a vector $\bfx$ we mean $\max_{i}{|x_i|}$. Also, when the notation $\|f\|_p$ is used for an integrable function $f$ we mean the usual norm on $L^p([0,1])$, and we write $\|f\|_{L^p(E)}$ to denote the $L^p(E)$ norm where $E$ is some real measurable set. As is conventional in analytic number theory, we write $e(z)$ for $e^{2 \pi i z}$. \par

Consider cubes $\grB^{(1)} \subset \R^{s_1}$ and $\grB^{(2)} \subset \R^{s_2}$, each having side length $2\eta$ for some positive $\eta \le 1$. Let $\bfgam^* \in [-1,1]^{s_1}$ denote the center of $\grB^{(1)}$ and $\bfrho^* \in [-1,1]^{s_2}$ denote the center of $\grB^{(2)}$. Following the same indexing system that was set up in the introduction, we write
\[  \bfgam^* = (\bfgam^*_1,\ldots,\bfgam^*_s) \quad\text{and}\quad \bfrho^* = (\bfrho^*_1,\ldots,\bfrho^*_s), \]
where $\bfgam^*_k = (\gamma_{k,1},\ldots,\gamma_{k,i_k}) \in \R^{i_k}$ and $\bfrho^*_k = (\rho_{k,1} , \ldots,\rho_{k,j_k}) \in \R^{j_k}$. Now, for a given set $E \subset \R$, we define the boxes
\[\grB^{(1)}_k (X,E) = \left\{\bfgam \in E^{i_k} : \left|\frac{\bfgam}{X^{1/2}} - \bfgam^*_k \right| < \eta  \right\},\]
and
\[\grB^{(2)}_k (X,E) = \left\{\bfrho \in E^{j_k} : \left|\frac{\bfrho}{X^{1/4}} - \bfrho^*_k \right| < \eta\right\}.\]
Then we define 
\begin{equation}\label{2.1}
    f_{k}(\a) = \sum_{ \substack{\bfx \in \grB^{(1)}_k (X,\Z) \\ \bfy \in \grB^{(2)}_k (X,\Z)}} e\left( \a F_{k}(\bfx; \bfy) \right),\quad f(\a) = \prod_{1 \le k \le s}f_k(\a),
\end{equation}
\begin{equation}\label{2.2}
    S_{k}(q,a) = \sum_{ \substack{\bfr \in (\Z/q\Z)^{i_k} \\ \bfs \in (\Z/q\Z)^{j_k} }} e \left( \frac{a}{q} F_{k}(\bfr; \bfs) \right),\quad S(q,a) = \prod_{1 \le k \le s}S_k(q,a),
\end{equation}
and
\begin{equation}\label{2.3}
    v_{k}(X,\b) = \int_{\substack{\bfgam \in \grB^{(1)}_k (X,\R) \\ \bfrho \in \grB^{(2)}_k (X,\R)}}e \left(\b F_{k}(\bfgam; \bfrho) \right) \d \bfgam \d \bfrho ,\quad I(X,\b) = \prod_{1 \le k \le s}v_{k}(X,\b).
\end{equation}
With this notation in hand we have by orthogonality 
\begin{equation}\label{2.4}
    R(X;\Phi;\grB^{(1)},\grB^{(2)}) = \int_{0}^{1} f(\a) \d \a.
\end{equation}
We begin by first estimating the exponential sums $f_k$ near rational points.
\begin{lemma}\label{lemma2.1}
    Given $a \in \Z$ and $q \in \N$ and any $\a \in \R$ we have that
    \[f_{k}(\a) = q^{-(i_k+j_k)} S_{k}(q,a) v_{k}(X,\a - a/q) + O \left(X^{i_k/2 + j_k/4 -1/4}\left(q + |q\a - a| X \right)\right). \]
\end{lemma}
\begin{proof}
    Let $k$ be fixed, then we may parametrize $\bfx \in \grB^{(1)}_k$ and $\bfy \in \grB^{(2)}_k$ by writing 
    \[\bfx = q \bfx' + \bfr,\quad \bfy = q\bfy' + \bfs.\]
    Then upon setting $\b = \a - a/q$ one obtains
    \begin{equation}\label{2.5}
        f_k(\a) = \sum_{(\bfr,\bfs) \in (\Z/q\Z)^{i_k+j_k}}e\left( \frac{a}{q}F_k(\bfr;\bfs) \right)\sum_{\substack{\bfx' \in \grB^{(1)}_k (\bfr) \\ \bfy' \in \grB^{(2)}_k (\bfs)}}e\left( \b F_k(q\bfx' + \bfr; q\bfy' + \bfs) \right),
    \end{equation}
    where 
    \[\grB^{(1)}_k (\bfr) = \left\{\bfx' \in \Z^{i_k}: \left|\frac{q\bfx' + \bfr}{X^{1/2}} - \bfgam^*_k\right| < \eta \right\},\]
    and
    \[\grB^{(2)}_k (\bfs) = \left\{\bfy' \in \Z^{j_k}: \left|\frac{q\bfy' + \bfr}{X^{1/4}} - \bfrho^*_k \right| < \eta  \right\}.\]
    Then by making use of the mean value theorem as in \cite[Lemma 4.2]{Davenport2005} and noting that $0 < i_k+j_k$, we obtain that inner sum over $\bfx'$ and $\bfy'$ in equation \ref{2.5} is equal to
    \[q^{-(i_k+j_k)}v_k(X,\b) + O\left(\frac{X^{i_k/2+j_k/4 - 1/4}}{q^{i_k + j_k - 1}}(1 + |\b|X) \right).\]
    Upon substituting this into (\ref{2.5}) we obtain the desired result.
\end{proof}
For any $\delta > 0$ and pair $a \in \Z$, $q \in \N$ we define
\[\grM_{\delta}(q,a) = \{\a \in [0,1): |\a - a/q| \le X^{\delta -1}\},\]
and so our major arcs are defined as
\[ \grM_\delta = \bigcup_{\substack{0 \le a \le q \le X^\delta \\ (a,q) = 1}}\grM_{\delta}(q,a),\]
with the minor arcs being $\grm_{\delta} = [0,1) \backslash \grM_{\delta}$. With our major/minor arcs established we immediately obtain, via Lemma \ref{lemma2.1} and equation (\ref{2.1}), a good estimate for the exponential sum $f(\a)$ over the major arcs.
\begin{corollary}\label{cor2.2}
    Given $a \in \Z$ and $q \in \N$ and any $\a \in \grM_\delta(q,a)$ we have that
    \[f(\a) = q^{-(s_1+s_2)} S(q,a)I(X,\a - a/q) + O \left(X^{s_1/2 + s_2/4  + 2\delta - 1/4} \right).\]
\end{corollary}
Next, we define
\[\grS(Q) = \sum_{q = 1}^{Q} \sum_{\substack{a = 1 \\ (a,q) = 1}}^{q} q^{-(s_1 + s_2)}S(q,a),\]
\[J(Q,X) = \int_{|\b| \le QX^{- 1}}I(X,\b) \d \b.\]
We conclude with a Lemma.
\begin{lemma}\label{lemma2.3}
    For $0 < \delta < 1/20$ the asymptotic formula
    \[\int_{\grM_\delta}f(\a) \d \a = \grS(X^\delta) J(X^\delta,1) X^{s_1/2 + s_2/4 - 1} +  o(X^{s_1/2 + s_2/4 - 1}) ,\]
    holds
\end{lemma}
\begin{proof}
    By applying Corollary \ref{cor2.2} to each interval in the major arcs and noting that the measure of the major arcs is $O(X^{3\delta - 1})$ we see that the major arc contribution is equal to
    \[\sum_{\substack{1 \le a \le q \le X^\delta \\ (a,q) = 1}} q^{-(s_1+s_2)}S(q,a)\int_{|\a - a/q| \le X^{\delta - 1}} I(X,\a - a/q)\d\a + o(X^{s_1/2 + s_2/4 - 1}).\]
    Via the substitution $\b = \a - a/q$ we note that the inner integral is now independent of $a,q$. Thus we see that
    \[\int_{\grM_\delta}f(\a) \d\a = \grS(X^\delta) J(X^\delta,X) +  o(X^{s_1/2 + s_2/4 - 1}).\]
    All that is left to show is that
    \[J(X^\delta,X) = J(X^{\delta},1)X^{s_1/2 + s_2/4 - 1}.\]
    One establishes this by first recalling (\ref{1.3}) and note that via a change of variables we have
    \[I(X,\b) = X^{s_1/2 + s_2/4} I(1,X\b),\]
    whence
    \begin{align*}
        J(X^\delta,X) &= X^{s_1/2 + s_2/4} \int_{|\b| \le X^{\delta - 1}} I(1,X\b) \d\b \\
        &=  X^{s_1/2 + s_2/4-1} \int_{|\b'| \le X^{\delta}} I(1,\b') \d\b' \\
        &=  X^{s_1/2 + s_2/4-1}  J(X^{\delta},1).
    \end{align*}
\end{proof}

\section{The Minor Arcs}\label{sec:minor}

Before we begin, some notation must be established. Given any particular $1 \le k \le s$, let $M_k$ be the largest coefficient in absolute value of the polynomial $F_k$, and then define 
\begin{equation} \label{3.1}
    M = \max_{1 \le k \le s}20M_k^2.
\end{equation}
This number will be useful later. Referencing our labeling scheme for the coefficients of $H_k^{(2)}$ as in (\ref{1.2}), we define the constant $\Delta_k = a_{k,1}^2-4a_{k,0}a_{k,2}$ and the polynomial
\[\delta_k(\bfv) = a_{k,1}H_k^{(2,1)}(\bfv)-2a_{k,0}H_k^{(2,2)}(\bfv).\]
It is important to note that by our hypothesis in Theorem \ref{thrm1.2}, we have that if $i_k = 2$ then $a_{k,0}$ and $\Delta_k$ are nonzero, and if $i_k = 1$ then we only have that $a_{k,0}$ is nonzero. It is important that we now develop robust notation which will allow us to concisely state our results regardless of which situation we find ourselves in, it is for this reason that we take some time to establish the following quantities. \par

For any $\textbf{b} \in \Z^2$ and $\bfxi \in \R^2$ we define 
\[
\chi_{k}^{(i)}(\textbf{b},\bfxi) = 
\begin{cases}
    \frac{b_1}{2a_{k,0}} - \Delta_k \xi_1, & \text{when }i = 1,  \\
    \frac{a_{k,1}b_1}{2a_{k,0} \Delta_k} - a_{k,1}\xi_1 - \frac{b_2}{\Delta_k} + \xi_2,  & \text{when }i = 2.
\end{cases}
\]
Also, for any set $E \subset \R$, real variables $Y,Z$, and $\bfgam \in \R^{i_k}$ we define the sets $B_k(Y,Z,E,\bfgam) \subset E^{i_k}$ to be
\[
\begin{array}{ll}
\{0\} & \text{when }i_k=0, \\
\left\{b_1 \in E: \left|\frac{b_1}{2 a_{k,0} Z^{1/2}} -  \gamma_1 \right| \le Y\right\} & \text{when }i_k = 1, \\
\left\{(b_1,b_2) \in E^2: \left|\frac{b_1}{2 a_{k,0} \Delta_k Z^{1/2}} -  \gamma_1 \right| \le Y, \left|\frac{b_2}{\Delta_k Z^{1/2}} -  \gamma_2\right| \le Y\right\} & \text{when }i_k=2.
\end{array}
\]
Here, whenever it is convenient to do so, we abbreviate $B_k(Y,Z,E,\bf0)$ to $B_k(Y,Z,E)$.\par

Now for the sake of brevity and generality, we define the following table of functions\begin{align*}
    \mathfrak{h}^{(1)}_k(\bfv) &= 
\begin{cases}
    H^{(4)}_k(\bfv), & i_k = 0, \\
    H^{(4)}_k(\bfv) - \frac{1}{4a_{k,0}}\left( H_k^{(2,1)}(\bfv) \right)^2, & i_k = 1, \\
    H^{(4)}_k(\bfv) - \frac{1}{4a_{k,0}} \left( H_k^{(2,1)}(\bfv) \right)^2 + \frac{1}{4 a_{k,0} \Delta_k} \left(\delta_k(\bfv)\right)^2, & i_k = 2,
\end{cases} \\
\mathfrak{h}^{(2)}_k(\bfv; \textbf{b},\bfxi) &= 
\begin{cases}
    0, & i_k = 0, \\
    - \chi_{k}^{(1)}(\textbf{b},\bfxi) H_k^{(2,1)}(\bfv), & i_k = 1, \\
    - \chi_{k}^{(1)}(\textbf{b},\bfxi) H_k^{(2,1)}(\bfv) +  \chi_{k}^{(2)}(\textbf{b},\bfxi)\delta_k(\bfv),  & i_k = 2,
\end{cases} \\
\mathfrak{g}^{(1)}_k(\bfu) &= 
\begin{cases}
    0, & i_k = 0, \\
    \frac{u_1^2}{4a_{k,0}}, & i_k = 1, \\
    \frac{u_1^2}{4a_{k,0}} - \frac{u_2^2}{4a_{k,0}\Delta_k},  & i_k = 2,
\end{cases} \\
\mathfrak{g}^{(2)}_k(\bfu;\textbf{b},\bfxi) &= 
\begin{cases}
    0, & i_k = 0, \\
    \chi_{k}^{(1)}(\textbf{b},\bfxi) u_1, & i_k = 1, \\
    \chi_{k}^{(1)}(\textbf{b},\bfxi) u_1 - \chi_{k}^{(2)}(\textbf{b},\bfxi) u_2,  & i_k = 2,
\end{cases} \\
\mathfrak{K}_k(\bfn,\bfxi) &=
\begin{cases}
    0, & i_k = 0, \\
    n_1\xi_1, & i_k = 1, \\
    n_1 \xi_1 + n_2 \xi_2, & i_k = 2.
\end{cases}
\end{align*}
We are now well enough equipped to state the following technical result.

\begin{proposition}\label{prop3.1}
For any real $\a$, and integer $1 \le k \le s$ we have that
\begin{equation*}
f_k(\a) = \frac{1}{\#B_{k}(1,1,\N)} \sum_{\textbf{b} \in B_{k}(1,1,\N)}\int_{[0,1]^{i_k}}  K_{k}(\bfxi)  g_{k}(\a;\textbf{b},\bfxi) h_{k}(\a;\textbf{b},\bfxi)  \d\bfxi,
\end{equation*}
where
\[h_{k}(\a;\textbf{b},\bfxi)= \sum_{\bfv \in \grB^{(2)}_k (X,\Z)} e \left( \a \mathfrak{h}^{(1)}_k(\bfv) +\mathfrak{h}^{(2)}_k(\bfv;\textbf{b},\bfxi) \right),\]
\[g_k(\a;\textbf{b},\bfxi)= \sum_{|\bfu| \le MX^{1/2} } e \left( \a \mathfrak{g}^{(1)}_k(\bfu) +\mathfrak{g}^{(2)}_k(\bfu;\textbf{b},\bfxi)  \right),\]
and
\[K_k(\bfxi) = \sum_{\bfn \in B_k(\eta, X,\Z , \bfgam_k^*)} e(\mathfrak{K}_k(\bfn,\bfxi)).\]
\end{proposition}
\begin{proof}
We begin by considering the case $i_k=2$. Using the substitutions $\bfv = \bfy$, 
\[u_1 = 2a_{k,0} x_1 + a_{k,1} x_2+ H_k^{(2,1)}(\bfv),\]
and
\[u_2 = \Delta_k x_2 + \delta_k(\bfv),\]
one obtains
\[F_{k}(\bfx;\bfy) = \frac{u_1^2}{4a_{k,0}} - \frac{u_2^2}{4a_{k,0}\Delta_k} + \mathfrak{h}^{(1)}_k(\bfv) .\]
This lets us rewrite the exponential sum for $f_{k}(\a)$ as long as we are careful to sum over the correct range for $u_1,u_2$. Let
\[ s_k(u_2,\bfv) = u_2 - \delta_k(\bfv), \]
and
\[t_k(\bfu,\bfv) = \Delta_k u_1 - \Delta_k H_k^{(2,1)}(\bfv) - a_{k,1} s_k(u_2,\bfv).\]
We may then define the correct range for $u_1,u_2$ by introducing the sets
\begin{align*}
    S_{k}^{(1)}(\bfv) &= \left\{ u_2 \in \Z:  s_k(u_2,\bfv) \equiv 0 \nmod{\Delta_k} \right\}, \\
    T_{k}^{(1)}(\bfv,u_2) &= \left\{ u_1 \in \Z: t_k(\bfu,\bfv) \equiv 0 \nmod{2 a_{k,0} \Delta_k} \right\}, \\
    S_{k}^{(2)}(\bfv) &= \left\{ u_2 \in \Z: \left|\frac{s_k(u_2,\bfv)}{\Delta_k X^{1/2}} -  \gamma^*_{k,2}\right| \le \eta  \right\}, \\
    T_{k}^{(2)}(\bfv,u_2) &= \left\{ u_1 \in \Z: \left|\frac{t_k(\bfu,\bfv)}{2a_{k,0}\Delta_k X^{1/2}} -\gamma^*_{k,1} \right| \le \eta  \right\}.
\end{align*}
Then setting $S_k(\bfv) = S_{k}^{(1)}(\bfv) \cap S_{k}^{(2)}(\bfv)$ and $T_k(\bfv,u_2) = T_{k}^{(1)}(\bfv,u_2) \cap T_{k}^{(2)}(\bfv,u_2)$ we may write
\begin{equation*}
    f_{k}(\a) = \sum_{\substack{\bfv \in \grB^{(2)}_k (X,\Z) }} \sum_{\substack{ u_2 \in S_k(\bfv)}} \sum_{ u_1 \in T_k(\bfv,u_2)} e \left( \a \left( \frac{u_1^2}{4a_{k,0}} - \frac{u_2^2}{4a_{k,0}\Delta_k} + \mathfrak{h}^{(1)}_k(\bfv)\right)\right).
\end{equation*}
\par

Given the definitions of $u_1,u_2$ and noting that $|\bfgam_k^*|,|\eta| \le 1$, one easily sees that $|\bfu| \le MX^{1/2}$ where $M$ is defined in (\ref{3.1}). This bound will be used later once we get rid of the pesky restrictions in the summand. We may immediately do this by using detector functions which pick up membership in $S_{k}^{(1)}(\bfv)$ and $T_{k}^{(1)}(u_2,\bfv)$. We do this by introducing the functions
\[
\kappa_{k}^{(1)}(\bfu,\bfv) = \frac{1}{2|a_{k,0} \Delta_k |}\sum_{b_1 = 1}^{2|a_{k,0} \Delta_k|} e \left( \frac{b_1}{2a_{k,0}\Delta_k} t_k(\bfu,\bfv) \right) ,
\]
\[
\kappa_{k}^{(2)}(u_2,\bfv) = \frac{1}{|\Delta_k|} \sum_{b_2 = 1}^{|\Delta_k|} e\left( \frac{b_2}{\Delta_k} s_k(u_2,\bfv) \right).
\]
We may similarly pick up membership in $S_{k}^{(2)}(\bfv)$ and $T_{k}^{(2)}(u_2,\bfv)$ by introducing the functions
\[
\lambda_{k}^{(1)}(\bfu,\bfv) = \int_{0}^{1} \sum_{ \left|\frac{n_1}{2a_{k,0} \Delta_k X^{1/2}} -\gamma^*_{k,1} \right| \le \eta } e \left( \left( n_1 - t_k(\bfu,\bfv) \right)\xi_1 \right) \d\xi_1 ,
\]
\[
\lambda_{k}^{(2)}(u_2,\bfv) = \int_{0}^{1} \sum_{\left|\frac{n_2}{\Delta_k X^{1/2}} -  \gamma^*_{k,2}\right| \le \eta} e \left( \left( n_2 - s_k(u_2,\bfv) \right) \xi_2\right) \d\xi_2.
\]
Thus $\kappa_{k}^{(2)}(u_2,\bfv)\lambda_{k}^{(2)}(u_2,\bfv) = \mathbbm{1}_{S_k(\bfv)}(u_2)$ and $\kappa_{k}^{(1)}(\bfu,\bfv)\lambda_{k}^{(1)}(\bfu,\bfv) = \mathbbm{1}_{T_k(u_2,\bfv)}(u_1)$. Plugging in these indicator functions we see that $f_{k}(\a)$ is equal to
\begin{equation*}
    \sum_{\substack{\bfv \in \grB^{(2)}_k (X,\Z) \\ |\bfu| \le MX^{1/2} } }\mathbbm{1}_{T_k(u_2,\bfv)}(u_1)\mathbbm{1}_{S_k(\bfv)}(u_2)  e \left( \a \left( \frac{u_1^2}{4a_{k,0}} - \frac{u_2^2}{4a_{k,0}\Delta_k} + \mathfrak{h}^{(1)}_k(\bfv)\right) \right).
\end{equation*}
By expanding $\mathbbm{1}_{T_k(u_2,\bfv)}(u_1)\mathbbm{1}_{S_k(\bfv)}(u_2)$ as exponential sums using our detector functions, we can interchange the order of summation and integration, making the summation ranges independent of each other. After this, we can group terms into separate exponential sums and arrive at the following equation.
\[f_{k}(\a) = \frac{1}{2|a_{k,0} \Delta_k|} \sum_{\substack{1 \le b_1 \le 2|a_{k,0}| \\ 1 \le b_2 \le |\Delta_k|}} \int_{0}^{1} \int_{0}^{1} K_k(\bfxi)g_k(\a;\textbf{b},\bfxi) h_{k}(\a;\textbf{b},\bfxi) \d\bfxi.\]
This establishes the case $i_k = 2$ of the Proposition. One may then, via a similar yet simpler proof, show that the case $i_k = 1$ follows and the case $i_k = 0$ is trivially true.
\end{proof}
\begin{lemma}\label{lemma3.2}
    Let $E \subset [0,1]$ be any Lebesgue measurable set and following our previously established indexing suppose that $\textbf{b} = (\textbf{b}_1,\ldots,\textbf{b}_s) \in \Z^{s_1}$ and $\bfxi = (\bfxi_1,\ldots,\bfxi_s) \in [0,1]^{s_1}$. In addition suppose that one has a bound of the form
    \[\Xi_E(\textbf{b};\bfxi) = \int_{E}\prod_{1 \le k \le s}|g_{k}(\a;\textbf{b}_k,\bfxi_k) h_{k}(\a;\textbf{b}_k,\bfxi_k)| \d\a \ll X^{\sigma},\]
    where $\sigma$ is fixed and the implicit constants are independent of $\textbf{b}$ and $\bfxi$. Then we have
    \[\int_{E} |f(\a)| \d\a \ll X^{\sigma}\log(X)^{s_1} .\]
\end{lemma}
\begin{proof}
    Set
    \[B = \prod_{1 \le k \le s}B_k(1,1,\N) ,\quad K(\bfxi) = \prod_{1 \le k \le s} |K_{k}(\bfxi_{k})|.\]
    We may apply Proposition \ref{prop3.1} to each polynomial $f_k$ where $1 \le k \le s$. Then, upon applying the triangle inequality and switching the order of summation and integration one obtains the relation
    \begin{equation}\label{3.2}
        \int_{E}|f(\a)| \d\a \le  \frac{1}{\#B}\sum_{\mathbf{b} \in B} \int_{[0,1]^{s_1}} \Xi(\textbf{b};\bfxi) K(\bfxi) \d\bfxi.
    \end{equation}
    Substituting the bound $\Xi(\textbf{b};\bfxi) \ll X^{\sigma}$ into (\ref{3.2}) and noting that the implicit constant is independent of $\textbf{b},\bfxi$ we obtain
    \[\int_{E}|f(\a)|\d\a \ll \frac{X^{\sigma}}{\#B}\sum_{\mathbf{b} \in B} \int_{[0,1]^{s_1}}  K(\bfxi) \d\bfxi.\]
    All that is left to prove is that
    \[\frac{1}{\#B}\sum_{\mathbf{b} \in B} \int_{[0,1]^{s_1}}  K(\bfxi) \d\bfxi \ll (\log X)^{s_1},\]
    this is however immediately clear upon recalling the definition of $K(\bfxi)$ and making use of the basic estimate
    \[\int_{0}^{1}\biggl| \sum_{|n| \le X}e(-n\xi) \biggr| \d\xi \ll \log X .\]
\end{proof}

As a consequence of Lemma \ref{lemma3.2}, if we can obtain a bound for $\Xi_{\grm_\delta}(\textbf{b};\bfxi)$, we may transfer this to a bound over the minor arcs. Before obtaining such a bound we first require a mean value estimate that allows us to deal with the cases when $1 \le s_1 \le 3$.
\begin{lemma}\label{lemma3.3}
    Let $G^{(1)} \in \Q$, $G^{(2)} \in \R$. Also, let $H^{(1)} \in \Q[y_1,y_2]$ be a non-degenerate homogeneous quartic, and $H^{(2)} \in \R[y_1,y_2]$ be a non-degenerate homogeneous quadratic. Then for positive numbers $P,Q$, $|\eta| \le 1$, and $\bfrho \in [-1,1]^2$, we define the exponential sums
    \[
    G(\a) = \sum_{|x| \le P}e(\a G^{(1)} x^2 + G^{(2)} x),
    \]
    \[
    H(\a) = \sum_{\left|\frac{\bfy}{Q}-\bfrho \right| \le \eta} e\left(\a H^{(1)}(\bfy) + H^{(2)}(\bfy) \right).
    \]
    Then for large $P,Q$ we have the following mean value bound
    \[
    \left\| G(\a)H(\a) \right\|_2^2 \ll PQ^{2 + \eps} + Q^{4}P^{\eps},
    \]
    where the implicit constant is only dependent on $G^{(1)}$ and the coefficients of $H^{(1)}$.
\end{lemma}
\begin{proof}
    Let $N$ be the smallest natural number such that $NG^{(1)}H^{(1)} \in \Z[y_1,y_2]$, then by a change of variables one has that
    \[\left\| G(\a)H(\a) \right\|_2^2 \le N\left\| G(N\a)H(N\a) \right\|_2^2,\]
    where, by orthogonality, the right hand side is bounded by $N$ times the number of integer solutions of 
    \[NG^{(1)}(x_1^2-x_2^2) = NH^{(1)}(\bfy_1) - NH^{(1)}(\bfy_2),\]
    where 
    \[|x_i| \le P \quad\text{and}\quad \left|\frac{\bfy_i}{Q} - \bfrho\right| \le \eta.\] 
    We then we split into two cases, if $|x_1| = |x_2|$, of which there are $O(P)$ many ways this can happen, then we count the number of solutions of $NH^{(1)}(\bfy_1) - NH^{(1)}(\bfy_2) = 0$. By orthogonality this quantity is equal to
    \begin{equation*}
    \Biggl\| \sum_{\left|\frac{\bfy}{Q}-\bfrho \right| \le \eta}e\left(\a NH^{(1)}(\bfy) \right) \Biggr\|^2 \d\a \ll Q^{2 + \eps},
    \end{equation*}
    where the bound comes from Theorem 2 of \cite{Wooley1999}. So in the case $|x_1| = |x_2|$ there are $O(PQ^{2 + \eps})$ choices for our variables. If $|x_1| \neq |x_2|$, then let $n = x_1^2 - x_2^2 \neq 0$, then there are at most $O(Q^4)$ choices for the $y_i$ and by the divisor estimate at most $O(P^\eps)$ choices for $x_1,x_2$. So in the case $|x_1| \neq |x_2|$ there are at most $O(Q^4 P^{\eps})$ choices for our variables. The fact that the implicit constant is dependent only on $G^{(1)}$ and the coefficients of $H^{(1)}$ is clear from our application of orthogonality and the way we chose $N$.
\end{proof}
With this mean value estimate, we are now ready to bound $\Xi_{\grm_\delta}(\textbf{b};\bfxi)$.
\begin{proposition}\label{prop3.4}
    Let $\textbf{b} = (\textbf{b}_1,\ldots,\textbf{b}_s) \in \Z^{s_1}$ and $\bfxi = (\bfxi_1,\ldots,\bfxi_s) \in [0,1]^{s_1}$. Then we have that
    \[\Xi_{\grm_\delta}(\textbf{b};\bfxi) \ll X^{s_1/2 + s_2/4 - 1 - \delta/32 + \eps},\]
    where the implicit constant is dependent only on the coefficients of $\Phi$ and $\eps$.
\end{proposition}
\begin{proof}
Referencing the definition of the exponential sums $g_k$ from Proposition \ref{prop3.1} one notes that in the case that $i_k = 2$ this exponential sum may be split and written as the product of two single variable exponential sums. Thus one may write
\begin{equation}\label{3.4}
    \prod_{1 \le k \le s}g_k(\a;\textbf{b}_k,\bfxi_k) = \prod_{1 \le i \le s_1}G_i(\a;\textbf{b},\bfxi),
\end{equation}
where, in reference to the notation of Lemma \ref{lemma3.3}, each of the exponential sums $G_i(\a;\textbf{b},\bfxi)$ satisfy the conditions on $G$ with the rational number $G^{(1)}$ dependent only on the coefficients of $\Phi$, the real number $G^{(2)}$ dependent on the coefficients of $\Phi$, $\textbf{b}$, and $\bfxi$. Finally, the length of the summation is $P = MX^{1/2}$. \par

By a similar thought process,  we may write
\begin{equation}\label{3.5}
    \prod_{1 \le k \le s}h_k(\a;\textbf{b}_k,\bfxi_k) = \left(\prod_{1 \le j \le s_2/2} H_j(\a;\textbf{b},\bfxi) \right) H_0(\a;\textbf{b},\bfxi),
\end{equation}
where instead of splitting exponential sums we sometimes combine exponential sums over a single variable into an exponential sum over two variables. In reference to the notation of Lemma \ref{lemma3.3}, each of the exponential sums $H_i(\a;\textbf{b},\bfxi)$ for $1 \le j \le s_2 /2$ satisfy the conditions on $H$ with the rational binary quartic $H^{(1)}$ dependent only on the coefficients of $\Phi$, the real binary quadratic $H^{(2)}$ dependent on the coefficients of $\Phi$, $\textbf{b}$, and $\bfxi$. Finally, the length of the sum being  $Q = X^{1/4}$ and some center $\bfrho = (\rho_1,\rho_2)$ coming from $\bfrho^*$. Additionally, the function $H_0(\a;\textbf{b},\bfxi)$ is either $1$ if $s_2$ is even, or it is equal to some $h_k(\a;\textbf{b}_k,\bfxi_k)$ where $j_k = 1$ if $s_2$ is odd. It is important to note here that with respect to $\a$ the arguments of all of these exponential sums are dependent only on the coefficients of $\Phi$. Setting $m = \min\{s_1,4\}$ we define
\[
S_1(\a;\textbf{b},\bfxi) = 
\begin{cases}
    \prod\limits_{1 \le j \le 8} H_j(\a; \textbf{b},\bfxi), & s_1 = 0, \\
    G_1(\a;\textbf{b},\bfxi) \prod\limits_{1 \le j \le 5}H_j(\a; \textbf{b},\bfxi), & s_1 = 1, \\
    \prod\limits_{1 \le i \le m}G_i(\a;\textbf{b},\bfxi) \prod\limits_{1 \le j \le 4-m}H_j(\a; \textbf{b},\bfxi), & s_1 \ge 2,
\end{cases}
\]
then in reference to equations (\ref{3.4}) and (\ref{3.5}) we implicitly define $S_2(\a;\textbf{b},\bfxi)$ to be the product of the remaining exponential sums such that it satisfies
\[S_1(\a;\textbf{b},\bfxi)S_2(\a;\textbf{b},\bfxi) =  \prod_{1 \le k \le s}g_k(\a;\textbf{b}_k,\bfxi_k) h_k(\a;\textbf{b}_k,\bfxi_k).\] \par
Thus, by an application of H\"{o}lder's inequality we obtain the bound
\begin{equation}\label{3.6}
    \Xi_{\grm_\delta}(\textbf{b},\bfxi) \le \left\| S_2(\a;\textbf{b},\bfxi) \right\|_{L^{\infty}(\grm_\delta)} \left\| S_1(\a;\textbf{b},\bfxi) \right\|_1.
\end{equation}
By our conditions on the number of variables from Theorem \ref{thrm1.2}, it must be the case that the function $S_2(\a;\textbf{b},\bfxi)$ is a non-empty product of exponential sums which contains at least one of the following three exponential sums,
\[G_{s_1}(\a;\textbf{b},\bfxi), H_{s_2/2}(\a;\textbf{b},\bfxi), \text{ or } H_0(\a;\textbf{b},\bfxi).\]
Via a direct application of either Weyl's inequality (see \cite[Lemma 2.4]{Vaughan1997}) or Theorem 1 from \cite{Wooley1999} we obtain the following bounds,

\[\sup_{\a \in \grm_\delta}|G_{s_1}(\a;\textbf{b},\bfxi)| \ll X^{1/2 - \delta/4 + \eps},\]

\[\sup_{\a \in \grm_\delta}|H_0(\a;\textbf{b},\bfxi)| \ll X^{1/4 - \delta/32 + \eps},\]
and
\[\sup_{\a \in \grm_\delta}|H_{s_2/2}(\a;\textbf{b},\bfxi)| \ll X^{1/2 - \delta/16 + \eps}.\]
We note here that, by our previous comment, the implicit constants arising from these pointwise bounds are only dependent on the coefficients of $\Phi$ and $\eps$. Regardless of which situation we find ourselves in we, at worst, obtain the bound
\begin{equation}\label{3.7}
    \left\| S_2(\a;\textbf{b},\bfxi) \right\|_{L^{\infty}(\grm_\delta)} \ll 
\begin{cases}
X^{s_2/4 - 4 - \delta/32 + \eps }, & s_1 = 0, \\
X^{s_2/4 - 5/2 - \delta/32 + \eps }, & s_1 = 1, \\
X^{s_1/2 + s_2/4 - 2 - \delta/32 + \eps }, & s_1 \ge 2.    
\end{cases}  
\end{equation} \par
All that is left is to bound $\left\| S_1(\a;\textbf{b},\bfxi) \right\|_1$. There are a few cases to consider. If $s_1 \geq 4$, by H\"{o}lder's inequality and Hua's Lemma (see \cite[Lemma 2.5]{Vaughan1997}) we have that $\left\| S_1(\a;\textbf{b},\bfxi) \right\|_1$ is bounded above by
\begin{equation*}
    \prod_{1 \le i \le 4} \left\|G_i(\a;\textbf{b},\bfxi)\right\|_4 \ll X^{1+\eps}.
\end{equation*}
If $s_1 = 3$, by H\"{o}lder's inequality, Hua's Lemma, and Lemma \ref{lemma3.3} we see that $\left\| S_1(\a;\textbf{b},\bfxi) \right\|_1$ is bounded above by
\begin{equation*}
    \left\|G_1(\a;\textbf{b},\bfxi)H_1(\a;\textbf{b},\bfxi)\right\|_2 \left\|G_2(\a;\textbf{b},\bfxi) \right\|_4 \left\|G_3(\a;\textbf{b},\bfxi) \right\|_4 \ll X^{1+\eps}.
\end{equation*}
If $s_1 = 2$, by H\"{o}lder's inequality and Lemma \ref{lemma3.3} we see that again $\left\| S_1(\a;\textbf{b},\bfxi) \right\|_1$ is bounded above by
\begin{equation*}
    \left\|G_1(\a;\textbf{b},\bfxi)  H_1(\a;\textbf{b},\bfxi) \right\|_2 \left\|G_2(\a;\textbf{b},\bfxi) H_2(\a;\textbf{b},\bfxi)\right\|_2 \ll X^{1+\eps},
\end{equation*}
If $s_1 = 1$, by H\"{o}lder's inequality, Lemma \ref{lemma3.3}, and Theorem 2 from \cite{Wooley1999} have that $\left\| S_1(\a;\textbf{b},\bfxi) \right\|_1$ is bounded above by
\[\|G_1(\a;\textbf{b},\bfxi) H_1(\a;\textbf{b},\bfxi)\|_2 \prod_{2 \le j \le 5} \|H_j(\a;\textbf{b},\bfxi)\|_8 \ll X^{2+\eps}.\]
Finally, if $s_1 = 0$, by H\"{o}lder's inequality and Theorem 2 from \cite{Wooley1999} have that $\left\| S_1(\a;\textbf{b},\bfxi) \right\|_1$ is bounded above by
\[\prod_{1 \le j \le 8} \|H_j(\a;\textbf{b},\bfxi)\|_8 \ll X^{3 + \eps}.\]
Combining these bounds we obtain,
\begin{equation}\label{3.8}
    \left\| S_1(\a;\textbf{b},\bfxi) \right\|_1 \ll 
    \begin{cases}
    X^{3 + \eps}, & s_1 = 0, \\
    X^{2 + \eps}, & s_2 = 1, \\
    X^{1 + \eps}, & s_2 \ge 2.
    \end{cases}
\end{equation}
We, again, note that all of of our implicit constants are only dependent on the coefficients of $\Phi$ and $\eps$. Combining the bounds (\ref{3.6}), (\ref{3.7}), and (\ref{3.8}) we obtain the desired result.
\end{proof}
Combining the result of Proposition \ref{prop3.4} with Lemma \ref{lemma3.2} with $E = \grm_\delta$, and Lemma \ref{lemma2.3} we deduce the following result.
\begin{corollary}\label{cor3.5}
    For any $0 < \delta < 1/20$ we have that
    \[R(X;\grB_1,\grB_2) = \mathfrak{S}(X^\delta)J(X^{\delta},1) X^{s_1/2 + s_2/4-1} + o(X^{s_1/2 + s_2/4 - 1} ) .\]
\end{corollary}

\section{The Singular Series}\label{sec:singser}
We begin by establishing the absolute convergence of the \emph{complete} singular series
\begin{equation}\label{4.1}
    \grS = \sum_{q = 1}^{\infty} A(q),\text{ where} \quad A(q)= \sum_{\substack{a = 1 \\ (a,q) = 1}}^{q} q^{-(s_1+s_2)}S(q,a).
\end{equation}
In order to achieve this we introduce the auxiliary functions, for $1 \le k \le s$ and fixed $\bfu \in \Z^{i_k}$, $\bfv \in \Z^{j_k}$, defined by
\[
x_{k}^{(1)}(\bfu,\bfv) = 2a_{k,0}u_1 + a_{k,1}u_2 + H_k^{(2,1)}(\bfv) \quad,\quad x_{k}^{(2)}(u_2,\bfv) = \Delta_k u_2 + \delta_k(\bfv),
\]
where, if $\bfu \in \Z$ we take $u_2 = 0$. Next, define 
\begin{align*}
    A_k &= 
    \begin{cases}
        1, & i_k = 0, \\
        4a_{k,0}, & i_k = 1, \\
        4 a_{k,0} \Delta_k, & i_k = 2,
    \end{cases} \\
    \tilde{\mathfrak{g}}_k(\bfu,\bfv) &= 
    \begin{cases}
        0, & i_k = 0, \\
        x_{k}^{(1)}(\bfu,\bfv)^2, & i_k = 1, \\
        \Delta_k x_{k}^{(1)}(\bfu,\bfv)^2 - x_{k}^{(2)}(u_2,\bfv)^2, & i_k = 2,
    \end{cases} \\
    \tilde{\mathfrak{h}}_k(\bfv) &=
\begin{cases}
    H^{(4)}_k(\bfv), & i_k = 0, \\
    4a_{k,0}H^{(4)}_k(\bfv) - \left( H_k^{(2,1)}(\bfv) \right)^2, & i_k = 1, \\
    4 a_{k,0} \Delta_k H^{(4)}_k(\bfv) - \Delta_k\left( H_k^{(2,1)}(\bfv) \right)^2 + \left( \delta_k(\bfv) \right)^2, & i_k = 2.
\end{cases} \\
\end{align*}
Then we define the exponential sum
\[
T_k(q,a) = \sum_{\bfv \in (\Z/q\Z)^{j_k}} e \left( \frac{a}{q} \tilde{\mathfrak{h}}_k(\bfv)\right) \sum_{\bfu \in (\Z/q\Z)^{i_k}} e\left( \frac{a}{q}\tilde{\mathfrak{g}}_k(\bfu,\bfv) \right).
\]
One then has a relation between $S_k$ from (\ref{2.2}) and the function $T_k.$
\begin{lemma}\label{lemma4.1}
    For any natural numbers $a,q$ one has that
    \begin{equation*}
    S_k(q,a) = \left(\frac{(a,A_k)}{A_k} \right)^{i_k + j_k}T_k\left(\frac{A_k q}{(a,A_k)},\frac{a}{(a,A_k)} \right).
\end{equation*}
\end{lemma}
\begin{proof}
    One begins by noting that
\begin{align*}
    S_k(q,a) &= \sum_{\substack{\bfu \in (\Z/q\Z)^{i_k} \\ \bfv \in (\Z/q\Z)^{j_k}}} e \left( \frac{a}{q A_k} \left( \Tilde{\grh}_k(\bfv) + \Tilde{\grg}_k(\bfu,\bfv) \right) \right).
\end{align*}
Upon setting 
\[a' = \frac{a}{(a,A_k)}\quad \text{and} \quad q' = \frac{A_k q}{(a,A_k)},\]
one sees from periodicity that
\[\sum_{\substack{\bfu \in (\Z/q\Z)^{i_k} \\ \bfv \in (\Z/q\Z)^{j_k}}} e \left( \frac{a}{q A_k} \left( \Tilde{\grh}_k(\bfv) + \Tilde{\grg}_k(\bfu,\bfv) \right) \right) = \left( \frac{q}{q'} \right)^{i_k+j_k} T_k(q',a').\]

\end{proof}
\par
We now note that both $S_k$ and $T_k$ satisfy the quasi-multiplicative property that when $(a,q) = (b,r) = (q,r) = 1$ we have
\begin{equation}\label{4.3}
    S_k(qr, ar+bq) = S_k(q,a)S_k(r,b),
\end{equation}
and
\begin{equation}\label{4.4}
    T_k(qr, ar+bq) = T_k(q,a)T_k(r,b).
\end{equation}
The proof of this is exactly the same as the one given for Lemma 2.10 in \cite{Vaughan1997}. Upon recalling (\ref{4.1}) and (\ref{2.2}) we see that (\ref{4.3}) and an application of the Chinese Remainder Theorem leads us to conclude that $A(q)$ is a multiplicative function, thus to establish the absolute convergence of the complete singular series we need only provide a good bound on $A(q)$ when $q$ is a prime power.
\begin{lemma}\label{lemma4.2}
    Given a prime $p$ and any $h \in \N$ we have that
    \[A(p^h) \ll
    \begin{cases}
        p^{-5/4}, & h = 1, \\
        p^{-9h/16}, & h \ge 2.
    \end{cases}
    \]
\end{lemma}
\begin{proof}
    First, we may suppose that $p > M$ or else the result is trivial since $2 \le p \le M = O(1)$ whence $p \asymp 1$. First, we proceed by bounding $T_k(p^h,a)$  for some $a$ which is relatively prime to $p$ and transferring this information to a bound on $S_k(p^h,a)$. \par
    
    Consider the case in which $i_k = 2$. Since $p > M$ it must be the case that $(p^h,A_k) = 1$ for all $1 \le k \le s$, thus for any fixed $u_2,\bfv$ we see that the set of residues
    \begin{equation*}
        \left\{2a_{k,0} u_1 + a_{k,1} u_2 + H_k^{(2,1)}(\bfv) \nmod{p^h} : 1 \le u_1 \le p^h \right\},
    \end{equation*}
     is in bijective correspondence with $\Z/p^h\Z$. Thus we have
    \begin{align*}
        \sum_{1 \le u_1 \le p^h}e \left(\frac{a\Delta_k}{p^h}\left( x_{k}^{(1)}(\bfu,\bfv) \right)^2 \right) &= \sum_{1 \le x \le p^h}e \left(\frac{a\Delta_k}{p^h}x^2 \right), \\
    \end{align*}
    note that this is independent of $u_2$ and $\bfv$. The same reasoning also leads to the conclusion that
    \[\sum_{1 \le u_2 \le p^h}e \left(-\frac{a}{p^h}\left(x_{k}^{(2)}(u_2,\bfv) \right)^2 \right) = \sum_{1 \le y \le p^h}e \left(-\frac{a}{p^h}y^2 \right),\]
    which is independent of $\bfv$. Combining these we see that
    \[\sum_{\bfu \in (\Z/p^h\Z)^2}e\left(\frac{a}{p}\tilde{\mathfrak{g}}_k(\bfu,\bfv) \right) = \sum_{1 \le x \le p^h}e \left(\frac{a\Delta_k}{p^h}x^2 \right)\sum_{1 \le y \le p^h}e \left(-\frac{a}{p^h}y^2 \right), \]
    since this is independent of $\bfv$ we conclude that $T_k(p^h,a)$ is bounded by
    \begin{equation*}
        \left|\sum_{1 \le x \le p^h}e \left(\frac{a\Delta_k}{p^h}x^2 \right) \right|  \left|\sum_{1 \le x \le p^h}e \left(-\frac{a}{p^h}x^2 \right) \right| \left|\sum_{\bfv \in (\Z/p^h\Z)^{j_k}} e \left( \frac{a}{p^h}\tilde{\mathfrak{h}}_k(\bfv) \right)\right|.
    \end{equation*} \par
    
    In the case $h = 1$, one may use the standard Gauss sum estimate to obtain square root cancellation in the first two exponential sums. We then utilize Corollary 2F of \cite{Schmidt2006} on the third exponential sum to obtain square root cancellation over one of the variables while trivially summing over any remaining variables. This yields the bound
    \[
    T_k(p,a) \ll  p^{1 + 3j_k/4}.
    \]
    In the case $h \ge 2$ we may apply Weyl's inequality to the first two exponential sums and either Weyl's inequality if $j_k =1$ or Theorem 1 from \cite{Wooley1999} if $j_k = 2$ to obtain the bound
    \[T_k(p^h,a) \ll p^{h(1 + 7j_k/8 + \eps)}.\]
    All in all we have established that for $i_k = 2$ we have the bound
    \begin{equation}\label{4.7}
        T_k(p^h,a) \ll  \begin{cases}
                    p^{i_k/2+ 3j_k/4}, & h = 1, \\
                    p^{h(i_k/2 + 7j_k/8 + \eps)}, & h \ge 2. 
                    \end{cases} 
    \end{equation}
    One may easily extend this bound to work for $i_k = 1$ or $0$ by noting that the only element of our argument that may change would be that we need only consider $i_k$ many quadratic exponential sums each of which obtains square root cancellation. \par
    
    Now we will show that the bound (\ref{4.7}) may also be applied to $S_k(p^h,a)$. Since $(p^h,A_k) = 1$ for all $1 \le k \le s$ and that $a$ is relatively prime to $p$, we have by the Chinese Remainder Theorem that there exists 
    \[b \in (\Z/p^h\Z)^{\times} \quad \text{and}\quad c_k \in \left(\Z/\frac{A_k}{(a,A_k)} \Z \right)^{\times} , \]
    such that
    \[\frac{a}{(a,A_k)} = b \frac{A_k}{(a,A_k)} + c_k p^h.\]
    Then by Lemma \ref{lemma4.1} and (\ref{4.4}) we conclude that
    \begin{align*}
        S_k(p^h,a) = \left(\frac{(a,A_k)}{A_k}\right)^{i_k+j_k} T_k \left(\frac{A_k}{(a,A_k)},c_k \right) T_k\left(p^h, b\right) \ll T_k(p^h,b),
    \end{align*}
    whence by (\ref{4.7})
    \begin{equation*}
        S_k(p^h,a) \ll  \begin{cases}
                    p^{i_k/2+ 3j_k/4}, & h = 1, \\
                    p^{h(i_k/2 + 7j_k/8 + \eps)}, & h \ge 2. 
                    \end{cases} 
    \end{equation*}
    Upon recalling (\ref{2.2}) and (\ref{4.1}) one may then conclude that
    \[A(p^h) \ll 
    \begin{cases}
    p^{1-s_1/2 - s_2/4},    & h = 1, \\
    p^{h(1 -s_1/2 - s_2/8 + \eps)}, & h \ge 2.
    \end{cases}\]
    One may derive that 
    \[\frac{s_1}{2} + \frac{s_2}{4} \ge \frac{9}{4} \quad \text{and} \quad \frac{s_1}{2} + \frac{s_2}{8} \ge \frac{13}{8},\] 
    for integer values of $s_1$, $s_2$ which satisfy our hypothesis concerning $s_1$ and $s_2$ in Theorem \ref{thrm1.2}, hence we conclude that
    \[
    A(p^h) \ll 
    \begin{cases}
        p^{-5/4}, & h = 1, \\
        p^{-h(5/8 - \eps)}, & h \ge 2.
    \end{cases}
    \]
    Which proves the result by simply substituting $5/8-\eps$ with $9/16$.
\end{proof}
With this bound established we are ready to show the absolute convergence of the singular series.
\begin{lemma}\label{lemma4.3}
    The complete singular series $\grS$ is absolutely convergent and additionally we have that
    \[|\grS - \grS(Q)| \ll Q^{-1/32}.\]
\end{lemma}
\begin{proof}
    By Lemma \ref{lemma4.2} there exists a constant $C$ dependent only on the coefficients of the polynomial $\Phi$ such that for every prime $p$, we have that
    \[\sum_{h = 1}^{\infty}p^{h/32}\left|A(p^h) \right| \le C p^{-1-1/16}.\]
    Hence by the Euler product
    \[\sum_{q>Q}|A(q)| \le \sum_{q = 1}^{\infty}\left(\frac{q}{Q} \right)^{1/32} |A(q)| \le Q^{-1/32} \prod_{p} \left(1 + Cp^{-1-1/16} \right) \ll Q^{-1/32}. \]
    Where the last inequality is due to the infinite product being known to be absolutely convergent.
\end{proof}
With this result we have by the same analysis as in section 2.6 in \cite{Vaughan1997} that
\[\grS = \prod_{p}\sigma(p),\]
where
\[\sigma_{p} = \lim_{h \to \infty} p^{h(1-s_1-s_2)} \#\{(\bfx,\bfy) \in (\Z/p^h\Z)^{s_1+s_2}: \Phi(\bfx,\bfy) \equiv 0 \nmod{p^h}\}. \]
Combining Lemma \ref{lemma4.3} with Corollary \ref{cor3.5}, one may apply Hensel's Lemma to derive that the existence of non-singular $p$-adic zeroes guarantees that $\sigma(p) > 0$ for all primes. 

We summarize our findings with a Corollary.
\begin{corollary}\label{cor4.3}
    For any $0 < \delta < 1/20$ we have that
    \[R(X;\grB_1,\grB_2) = J(X^{\delta},1) \prod_{p}\sigma(p) X^{s_1/2 + s_2/4 - 1} + o(X^{s_1/2 + s_2/4 - 1}) .\]
    If we additionally assume $p$-adic solubility of $\Phi$ then $\prod_{p}\sigma(p) > 0$.
\end{corollary}

\section{The Singular Integral}\label{sec:singint}
Here, we establish the convergence of the \emph{complete} singular integral
\[\sigma_{\infty}(\grB^{(1)},\grB^{(2)}) = \lim_{Q \to \infty}J(Q,1) = \lim_{Q \to \infty} \int_{|\b| \le Q}I(1,\b) \d\b.\]
Thus far we have not specified the center, $\bfxi^* := (\bfgam^*, \bfrho^*)$, nor the side length, $2\eta$, of the cubes $\grB^{(1)}$,  $\grB^{(2)}$. Here, we show that there exists a choice of these free variables such that one may show the convergence of the \emph{complete} singular integral. Our strategy will be the same as outlined in chapter 16 of \cite{Davenport2005}.

\begin{lemma}\label{lemma5.1}
    Suppose that $\Phi$ has non-singular real solutions. Then there exists a cube $\grB =\grB^{(1)} \times \grB^{(2)} \subset \R^{s_1 + s_2}$ with center $\bfxi^*$ and side length $2\eta^*$ such that the limit
    \[\sigma_{\infty}(\grB^{(1)},\grB^{(2)}) = \lim_{Q \to \infty} \int_{|\b| \le Q} I(1,\b) \d\b,\]
    exists and is strictly positive, hence
    \[J(Q,1) = \sigma_{\infty}(\grB^{(1)},\grB^{(2)}) + o(1).\]
\end{lemma}
\begin{proof}
    Let $\bfxi^*$ be a non-singular real solution to $\Phi$, then define
\[c_i = \frac{\partial \Phi}{\partial \xi_i}(\bfxi^*),\]
for each $1 \le i \le s_1+s_2$. Without loss of generality we may suppose that $c_1 \neq 0$ then by Taylor's Theorem
\[\Phi(\bfxi^* + \bfxi) =  \sum_{1 \le i \le s_1+s_2}c_i \xi_i + P(\bfxi),\]
where $P$ is a polynomial which is $O(|\xi|^2)$ near $\xi = 0$. Call $\zeta = \Phi(\bfxi^* + \bfxi)$, then by the implicit function Theorem there then exists a strictly positive real number $\eta$ for which on the neighborhood $|\bfxi| < \eta$ the variable $\xi_1$ can be defined in terms of $\xi_2,\ldots,\xi_{s_1+s_2}$ and $\zeta$. From the above we see that this function is of the form
\[\xi_1 = \frac{\zeta}{c_1} - \sum_{2 \le i \le s_1+s_2}\frac{c_i}{c_1} \xi_i + P_1(\zeta,\xi_2,\ldots,\xi_{s_1+s_2}), \]
where $P_1$ is a multiple power series with terms of degree at least two. Taking partial derivatives with respect to $\zeta$ yields
\[\frac{\partial\xi_1}{\partial \zeta} = \frac{1}{c_1} + \frac{\partial P_1}{\partial \zeta}(\zeta,\xi_2,\ldots,\xi_{s_1+s_2}).\]
By continuity there exists a function $\sigma(\eta)$ which goes to zero as $\eta$ goes to zero which satisfies the statement
\[|\bfxi| \le \eta \text{ implies that } |\zeta(\bfxi)| \le \sigma(\eta).\]
Thus, we choose $\eta^*$ to be sufficiently small such that
\[
\left|\frac{\partial P_1}{\partial \zeta}(\zeta,\xi_2,\ldots,\xi_{s_{1}+s_2})\right|< \frac{1}{2|c_1|},
\]
holds for $|\zeta| \le \sigma(\eta^*)$ and $|\xi_i| \le \eta^*$ for $2 \le i \le s_1 + s_2$. \par
Thus, having chosen the center of our cube $\grB$ to be $\bfxi^*$ and the side length to be $2\eta^*$ we have by our above analysis that
\begin{align*}
    \sigma_{\infty}(\grB^{(1)},\grB^{(2)}) &= \lim_{Q \to \infty}  \int_{[-\eta^*,\eta^*]^{s_1+s_2}} \frac{\sin (2 \pi Q \Phi(\bfxi^* + \bfxi))}{\pi \Phi(\bfxi^* + \bfxi)} \d\bfxi \\
    &= \lim_{Q \to \infty}  \int_{-\sigma(\eta^*)}^{\sigma(\eta^*)} \frac{\sin (2 \pi Q \zeta)}{\pi \zeta}V(\zeta) \d\zeta.
\end{align*}
Where
\[V(\zeta) = \int_{\grB(\eta^*,\zeta)}\left( \frac{1}{c_1} + \frac{\partial P_1}{\partial \zeta}(\zeta,\xi_2,\ldots,\xi_{s_1+s_2})\right) \d\xi_2 \cdots \d\xi_{s_1+s_2},\]
and $\grB(\eta^*,\zeta)$ is the set of $\xi_2,\ldots,\xi_{s_1 + s_2}$ satisfying $|\xi_i| < \eta$ for which 
\[|\xi_1(\zeta,\xi_2,\ldots,\xi_{s_1+s_2})| < \eta^*.\] 
It is clear that $V(\zeta)$ is of bounded variation because it has both left and right bounded derivatives everywhere. Thus by Fourier's integral Theorem the limit as $Q \to \infty$ exists and we conclude that
\[\sigma_{\infty}(\grB^{(1)},\grB^{(2)}) = V(0).\]
The value $V(0)$ is strictly positive by construction since 
\[\left|\frac{1}{c_1} + \frac{\partial P_1}{\partial \zeta}(0,\xi_2,\ldots,\xi_{s_1+s_2})\right| > \frac{1}{2|c_1|},\] 
for $(\xi_2,\ldots,\xi_{s_1+s_2}) \in \grB(\eta^*,0)$ and $\text{mes}\{\grB(\eta^*,0)\} \gg 1$. 
\end{proof}
Upon combining the result of Lemma \ref{lemma5.1} with Corollary \ref{cor4.3} and setting $\delta < 1/20$ we have finally proven Theorem \ref{thrm1.2}.

\section{Further Problems}\label{sec:furprob} 

As referenced in section \ref{sec:intro}, there are many possible weighted forms one could investigate. Here, we outline a general class of problems that may be tackled in a manner similar to the one we discuss in this paper. Let $d \ge 2$ be given, then if we let $H^{(n)}$ denote a form of degree $n$ one may investigate the zeros of the weighted $2d$ form
\[\Phi_{2d}(\bfx;\bfy) = H^{(2)}(\bfx) + \sum_{1 \le i \le s_1}x_i H_i^{(d)}(\bfy) + H^{(2d)}(\bfy).\]
Note that $\Phi_{2d}$ is a general weighted $2d$ form in which the variable weights are either $1$ or $d$, up to a relabeling of variables. The most pressing issue is non-trivially bounding the minor arcs. We start by first setting $\bfv = \bfy$ and finding a suitable non-singular linear substitution $\bfx \mapsto \Tilde{\bfu}$ such that
\[H^{(2)}(\bfx) + \sum_{1 \le i \le s_1}x_i H^{(d)}(\bfy) = \frac{1}{Q}\left( \sum_{1 \le i \le s_1}\Tilde{a}_i \Tilde{u}_i^2 +\Tilde{u}_i\Tilde{H}_i^{(d)}(\bfv)\right),\]
where $Q,\Tilde{a}_1,\ldots,\Tilde{a}_{s_1} \in \Z\backslash\{0\}$ and the $\Tilde{H}_i^{(d)}$ are integral linear combinations of the $H_i^{(d)}$. Upon making the substitutions
\[u_i = 2\Tilde{a}_i\Tilde{u}_i + \Tilde{H}_i^{(d)}(\bfv),\quad A = \prod_{1 \le i \le s_1} \Tilde{a}_i, \quad a_i = \frac{A}{\Tilde{a}_i},\] 
and
\[\Tilde{H}^{(2d)}(\bfv) = 4AQ H^{(2d)}(\bfv) - \sum_{1 \le i \le s_1} a_i \left(\Tilde{H}^{(d)}(\bfv)\right)^2,\]
one has that
\[\Phi_{2d}(\bfx;\bfy) = \frac{1}{4AQ}\left( \sum_{1 \le i \le s_1}a_i u_i^2 + \Tilde{H}^{(2d)}(\bfv) \right).\]
Hence, counting zeros of $\Phi_{2d}(\bfx;\bfy)$ is equivalent to counting zeros of
\[\sum_{1 \le i \le s_1}a_i u_i^2 + \Tilde{H}^{(2d)}(\bfv),\]
where given a particular choice of $\bfv$ there are certain congruence conditions on the variables $\bfu$ which depend on the coefficients of $H^{(2)}$ and $H_i^{(d)}(\bfv)$.\par

It seems reasonable that the arguments present in section \ref{sec:minor} should allow one to understand the minor arcs in this problem given bounds on an integral of the shape
\[I(\grm) = \int_{\grm} |g(\a)^{s_1} h(\a)|d\a\]
where 
\[g(\a) = \sum_{|x| \le X^{1/2}} e(\a x^2),\quad h(\a) = \sum_{|\bfv| \le X^{1/2d}}e(\a \Tilde{H}^{(2d)}(\bfv)),\]
and some suitable definition of the minor arcs $\grm$. Assuming some reasonable non-singularity conditions on the form $\Tilde{H}^{(2d)}(\bfv)$ one could apply H\"{o}lder's inequality, classical exponential sum bounds, and Lemma 4.3 from \cite{Birch1962} to establish that
\[I(\grm) = o(X^{s_1/2 + s_2/4 - 1})\]
whenever $s_1$ and $s_2$ satisfy 
\[\frac{s_1}{2} + \frac{s_2}{2^{2d-1}(2d-1)} > 2.\]
However, there are some technical challenges with this approach. Thus, in this paper, we decided to focus on cases where the polynomial $\Tilde{H}^{(2d)}(\bfv)$ can be decomposed as a sum of binary forms. This approach enables us to elucidate the key concepts without delving too deeply into intricate technical details. \par

We finish by noting that the problem of general weighted forms appears to be beyond the scope of our current techniques as we require some way to separate the variables of large weight. To see this consider the case of a weighted $12$ form $\Phi(\bfx;\bfy;\bfz)$ where the variable $\bfx$ has weight $4$, $\bfy$ has weight $3$, and $\bfz$ has weight $2$. This weighted $12$ form contains integral linear combinations of forms of the shape
\[H^{(3)}(\bfx)H^{(4)}(\bfy) H^{(6)}(\bfz).\]
The main tool we have been using so far is completing the square which allows us to separate the variables of large weight, which is not applicable to the above form. It appears that new ideas will be required to make further progress.

\bibliographystyle{amsbracket}
\providecommand{\bysame}{\leavevmode\hbox to3em{\hrulefill}\thinspace}

\end{document}

%% file: abbrv.tex
\newtheorem{theorem}{Theorem}[section]
\newtheorem{lemma}[theorem]{Lemma}
\newtheorem{corollary}[theorem]{Corollary}

\newtheorem{proposition}[theorem]{Proposition}

\theoremstyle{definition}
\newtheorem{definition}[theorem]{Definition}

\theoremstyle{remark}

\numberwithin{equation}{section}

\newcommand{\nmod}[1]{\,\,\text{\rm mod}\,\,#1}

\def\bfn{{\mathbf n}}

\def\bfr{{\mathbf r}}
\def\bfs{{\mathbf s}}

\def\bfu{{\mathbf u}}
\def\bfv{{\mathbf v}}
\def\bfw{{\mathbf w}}
\def\bfx{{\mathbf x}}
\def\bfy{{\mathbf y}}
\def\bfz{{\mathbf z}}

\def\N{{\mathbb N}}  
\def\R{{\mathbb R}}
\def\Z{{\mathbb Z}}\def\Q{{\mathbb Q}}

\def\grB{{\mathfrak B}}

\def\grg{{\mathfrak g}}
\def\grh{{\mathfrak h}}

\def\grm{{\mathfrak m}}\def\grM{{\mathfrak M}}

\def\grS{{\mathfrak S}}
\def\grB{{\mathfrak B}}

 \def\grX{{\mathfrak X}}
 \def\grY{{\mathfrak Y}}

\def\a{{\alpha}}  
\def\b{{\beta}}  
\def\gam{{\gamma}} 
\def\bfgam{{\boldsymbol \gam}}

\def\bfxi{{\boldsymbol \xi}}

\def\bfrho{{\boldsymbol \rho}}

\def\eps{\varepsilon}

\def\le{\leqslant} \def\ge{\geqslant}

\def\d{{\,{\rm d}}}